\documentclass[11pt,twoside,reqno]{amsart}
\usepackage[all]{xy}   
\usepackage {amssymb,latexsym,amsthm,amsmath,setspace,bm}
\usepackage[colorlinks, linkcolor=blue, citecolor=red]{hyperref}
\usepackage{graphicx}
\usepackage{mathtools}
\usepackage{xcolor}
\usepackage[utf8]{inputenc}
\usepackage{tikz}
\usepackage{enumitem}

\newcommand{\secref}[1]{Section~\ref{#1}}
\newcommand{\subsecref}[1]{Subsection~\ref{#1}}
\newcommand{\thmref}[1]{Theorem~\ref{#1}}
\newcommand{\lemref}[1]{Lemma~\ref{#1}}
\newcommand{\remref}[1]{Remark~\ref{#1}}

\newcommand{\corref}[1]{Corollary~\ref{#1}}

\makeatletter
\def\imod#1{\allowbreak\mkern10mu({\operator@font mod}\,\,#1)}
\makeatother

\newtheorem{theorem}{Theorem}[subsection]
\newtheorem{lemma}[theorem]{Lemma}
\newtheorem{corollary}[theorem]{Corollary}

\newtheorem*{theorem*}{Theorem}

\theoremstyle{definition}
\newtheorem{remark}[theorem]{Remark}

\newtheorem{example}[theorem]{Example}
\newtheorem{definition}[theorem]{Definition}

\numberwithin{equation}{section}

\newcommand{\ignore}[1]{}

\newcommand{\nc}{\newcommand}

\nc{\Cal}{\cal} \nc{\Xp}[1]{X^+(#1)} \nc{\Xm}[1]{X^-(#1)}
\nc{\on}{\operatorname} \nc{\Z}{{\mathbb{Z}}} \nc{\J}{{\cal J}} \nc{\C}{{\mathbb{C}}} \nc{\Q}{{\bold Q}} \nc{\R}{{\mathbb{R}}}
\nc{\K}{\bold{\kappa}}

\nc{\F}{\operatorname{\boF}}
\nc{\pr}{\operatorname{\bold{pr}}}
\nc{\ev}{\operatorname{{ev}}}
\nc{\ch}{\operatorname{{ch}}}
\nc{\hook}{\lceil} 
\nc{\HLn}{\noindent\makebox[\linewidth]{\rule{\textwidth}{1pt}}}

\nc{\N}{{\Bbb N}} \nc\boa{\bold a} \nc\bob{\bold b} \nc\boc{\bold c} \nc\bod{\bold d} \nc\boe{\bold e} \nc\bof{\bold f} \nc\bog{\bold g}
\nc\boh{\bold h} \nc\boi{\bold i} \nc\boj{\bold j} \nc\bok{\bold k} \nc
\bol{\bold l} \nc\bom{\bold m} \nc\bon{\bold n} \nc\boo{\bold o}
\nc\bop{\bold p} \nc\boq{\bold q} \nc\bor{\bold r} \nc\bos{\bold s} \nc\boT{\bold t} \nc\boF{\bold F} \nc\bou{\bold u} \nc\bov{\bold v}
\nc\bow{\bold w} \nc\boz{\bold z} \nc\boy{\bold y} \nc\ba{\bold A} \nc\bb{\bold B} \nc\bc{\bold C} \nc\bd{\bold D} \nc\be{\bold E} \nc\bg{\bold
	G} \nc\bh{\bold H} \nc\bi{\bold I} \nc\bj{\bold J} \nc\bk{\bold K} \nc\bl{\bold L} \nc\bM{\bold M} \nc\bn{\bold N} \nc\bo{\bold O} \nc\bp{\bold
	P} \nc\bq{\bold Q} \nc\br{\bold R} \nc\bs{\bold S} \nc\bt{\bold T} \nc\bu{\bold U} \nc\bv{\bold V} \nc\bw{\bold W} \nc\bz{\bold Z} \nc\bx{\bold
	x} \nc\KR{\bold{KR}} \nc\rk{\bold{rk}} \nc\het{\text{ht }}\nc\wt{\text{wt }}

\nc\udi{\underline i} \nc\udj{\underline j}

\nc\fn{{fin}}  \nc\af{{aff}}  \nc\tr{{tor}} \nc\btilde{\bold{\tilde{\bold{H}}}}

\nc{\mpp}{\rotatebox[origin=c]{180}{\pm}}

\nc\eps{\epsilon}
\nc\toa{\tilde a} \nc\tob{\tilde b} \nc\toc{\tilde c} \nc\tod{\tilde d} \nc\toe{\tilde e} \nc\tof{\tilde f} \nc\tog{\tilde g} \nc\toh{\tilde h}
\nc\toi{\tilde i} \nc\toj{\tilde j} \nc\tok{\tilde k} \nc\tol{\tilde l} \nc\tom{\tilde m}  \nc\ton{\tilde n} \nc\too{\tilde o} \nc\toq{\tilde q}
\nc\tor{\tilde r} \nc\tos{\tilde s} \nc\toT{\tilde t} \nc\tou{\tilde u} \nc\tov{\tilde v} \nc\tow{\tilde w} \nc\toz{\tilde z}
\providecommand{\keywords}[1]
{
	\small	
	\textbf{Keywords---} #1
}

\begin{document}
	\setcounter{section}{0}
	\setcounter{tocdepth}{1}
	
	\title[Tensor Decomposition of Demazure Crystals]{Tensor Decomposition of Demazure Crystals for Symmetrizable Kac-Moody Lie algebras }
	\author[Divya Setia]{Divya Setia}
	\address{Institute of Mathematics, Polish Academy of Sciences, Kraków 31-027, Poland}
	\email{divyasetia01@gmail.com}
	\date{}


	
	\begin{abstract} 
		We study the tensor product of Demazure crystals for symmetrizable Kac-Moody Lie algebras. It is not necessary that the tensor product of Demazure crystals is isomorphic to a disjoint union of Demazure crystals. In this paper, we provide necessary and sufficient conditions for the decomposition of the tensor product of Demazure crystals as a disjoint union of Demazure crystals. Our results are the generalization of the results proved by Anthony Joseph and Takafumi Kouno. As an application, we obtain a sufficient condition when the product of Demazure characters is a linear combination of Demazure characters with nonnegative integer coefficients. In particular, we obtain a partial solution for the key positivity problem.    
	\end{abstract}
	
	\maketitle
	\section{INTRODUCTION}
	Let $\mathfrak{g}$ be a symmetrizable Kac-Moody Lie algebra and denote by $U_q(\mathfrak{g})$ its associated quantized universal enveloping algebra. In \cite{Kq-analogue}, Kashiwara introduced the notion of crystal basis $B(\lambda)$ viewed as bases at $q=0$ of an integrable irreducible highest weight $U_q(\mathfrak{g})$-module $V(\lambda)$ generated by the highest weight vector $v_{\lambda}$. An upper normal crystal $B(\infty)$ was also introduced in \cite{Kq-analogue} and it was proved that $B(\infty)$ is a highest weight crystal. $B(\infty)$ is a crystal analogue of a Verma module. Crystal basis is an important combinatorial tool to study representations of symmetrizable Kac-Moody Lie algebras.
	
	Certain subspaces of $V(\lambda)$ can be explored using crystal basis. For a Weyl group element $w$, the Demazure module $V_w(\lambda)$ is a module of a Borel subalgebra of $\mathfrak{g}$ and generated by the extremal weight vector $v_{w\lambda}$. A crystal basis $B_w(\lambda)\subseteq B(\lambda)$ of the Demazure module $V_w(\lambda)$ was given in \cite{K1demazure} and known as the Demazure crystal. Similarly the limiting Demazure crystal $B_w(\infty)\subseteq B(\infty)$ was also introduced in \cite{K1demazure}. The character formula for Demazure modules was proved by Anthony Joseph in \cite{Anthonydemazureformula} and Kashiwara gave a new combinatorial proof of the Demazure character formula in \cite{K1demazure} using Demazure crystals. 
	
	The tensor product of irreducible highest weight crystals $B(\lambda)\otimes B(\mu)$ can be defined using a simple combinatorial formula. Using the decomposition of the tensor product of irreducible highest weight modules $V(\lambda)\otimes V(\mu)$ (see \cite{L2, L1}), it was proved that the tensor product of corresponding highest weight crystals $B(\lambda)\otimes B(\mu)$ decomposes as a disjoint union of connected highest weight crystals. A similar decomposition problem can also be explored for the tensor product of Demazure crystals $B_v(\lambda)\otimes B_w(\mu)$, where $v,w$ are elements of the Weyl group $W$ and $\lambda,\mu \in P^+$ the set of dominant weights. In general, it is not necessary that the tensor product of Demazure crystals can be decomposed as a disjoint union of Demazure crystals. The tensor product of Demazure crystals $B_e(\lambda)\otimes B_w(\mu)$ that is $b_{\lambda} \otimes B_w(\mu)$ was studied in \cite{LLM} for finite dimensional Lie algebras and it was proved that $B_e(\lambda)\otimes B_w(\mu)$ decomposes as a disjoint union of Demazure crystals using L-S paths.
	
	In \cite{Anthony}, Anthony Joseph generalized the result for symmetrizable Kac-Moody Lie algebras and proved that both $B_e(\lambda)\otimes B_w(\mu)$ and $B_e(\lambda)\otimes B_w(\infty)$ decomposes as a disjoint union of Demazure crystals using Kashiwara crystal basis. In \cite{Assaf}, authors introduced the notion of extremal subsets that is a subset $X\subseteq B(\lambda)$ is said to be extremal if for any $i$-string $S\subset B(\lambda)$ with highest weight element $b$, its intersection $S\cap X$ is either $\emptyset, S$ or $\{b\}$. All Demazure crystals are extremal subsets (see \cite{K1demazure}) but not all extremal subsets are Demazure crystals. Using the results proved by Anthony Joseph in \cite{Anthony}, we provide a complete characterization for when the tensor product of Demazure crystals $B_v(\lambda)\otimes B_w(\mu)$ decomposes as a disjoint union of Demazure crystals in the following theorem.
	 	\begin{theorem}\label{thm 1}
	 	For $\lambda,\mu \in P^+$ and $v,w\in W$, the following statements are equivalent.
	 	\begin{enumerate}
	 		\item $B_v(\lambda)\otimes B_w(\mu)$ is isomorphic to a disjoint union of Demazure crystals.
	 		\item $v_{\min}^{\lambda}\in \langle\{s_i\,|\, \langle w\mu,\alpha_i^\vee\rangle \leq 0\}\rangle$.
	 		\item $B_v(\lambda)\otimes B_w(\mu)$ is an extremal subset of $B(\lambda)\otimes B(\mu)$.
	 	\end{enumerate}
	 \end{theorem}
Here $v_{\min}^\lambda$ denotes the minimal length representative of the coset $vW_{\lambda}$, where $W_{\lambda}:= \{w\in W|\, w\lambda = \lambda\}$. \thmref{thm 1} is a generalization of the results proved in \cite{Knouo, Assaf} for finite dimensional Lie algebras to symmetrizable Kac-Moody Lie algebras. 

For $i\in I = \{1,2,\cdots n\}$, let $e_i$ and $f_i$ be the raising and lowering operator respectively and an element $b$ of a crystal $B$ is said to be primitive if $e_i(b) =0,\, \forall i\in I$. For $\mu \in P^+ \cup \{\infty\}$, denote by $B(\mu)^\lambda$, the set of all $b\in B(\mu)$ such that $b_{\lambda} \otimes b$ is primitive in $B(\lambda)\otimes B(\mu)$ and define $B_w(\mu)^\lambda = B(\mu)^\lambda \cap B_w(\mu)$. Using the notion of primitive elements we describe the Demazure crystals that appear in the decomposition if $B_v(\lambda)\otimes B_w(\mu)$ is isomorphic to a disjoint union of Demazure crystals.
 \begin{theorem}\label{thm 2}
 	Let $\lambda, \mu \in P^+,\,v,w\in W, $ and $s_{i_1}\cdots s_{i_k}$ be a reduced expression of $v_{\min}^\lambda$ in $W$. Assume that $v_{\min}^{\lambda}\in \langle \{s_i\, | \langle w\mu,\alpha_{i}^{\vee}\rangle\leq 0\} \rangle$, then for all $b\in B_w(\mu)^\lambda$ there exist $u(b,v)\in W$ such that
 	$$B_v(\lambda)\otimes B_w(\mu)\cong \coprod_{b\in B_w(\mu)^{\lambda}}B_{u(b,v)}(\lambda+wt(b)).$$ 
 \end{theorem}	
As an application of \thmref{thm 1}, we provide a criterion when the product of Demazure characters is a linear combination of Demazure characters with nonnegative integer coefficients. For $i\in I$, define the Demazure operator $\Delta_i$ on $P$ by $$\Delta_i e^\mu = \frac{e^\mu - e^{s_i \mu -\alpha_i}}{1-e^{-\alpha_i}}.$$ Given a reduced decomposition $s_{i_1}\cdots s_{i_k}$ of $w$ in $W$, set $\Delta_w = \Delta_{i_1}\cdots \Delta_{i_k}$ then it follows from \cite{K1demazure, L2} that $\ch B_w(\mu) = \Delta_w e^\mu.$ Using \thmref{thm 1} and \thmref{thm 2}, we obtain the following result.
\begin{theorem}\label{thm3}
		For $\lambda,\mu \in P^+$ and  $v,w\in W$, assume that $v_{\min}^{\lambda} \in \langle \{s_i\,| \langle  w\mu, \alpha_i^{\vee} \rangle\leq 0 \}\rangle$. Then the following statements hold.
		\begin{enumerate}
			\item $\ch (B_v(\lambda)\otimes B_w(\mu)) = \Delta_v(e^\lambda(\Delta_w)e^\mu).$
			\item $\Delta_v(e^\lambda(\Delta_w)e^\mu)$ is a linear combination of some Demazure characters with nonnegative integer coefficients.
		\end{enumerate} 
\end{theorem}
In particular, if we assume $\mathfrak{g}$ to be finite dimensional simple Lie algebra of type $A$ then the character of a Demazure crystal is same as the key polynomial and we obtain a sufficient condition for the key positivity problem. In \cite{RS}, it was proved that a product of key polynomials is a linear combination of key polynomials with integer coefficients. Therefore, the key positivity problem is to prove that the coefficients are nonnegative. Using \thmref{thm3}, we obtain a partial solution for the key positivity problem without using L-S paths similar to the results proved by Kouno in \cite{Knouo}.  

We also study the tensor product of Demazure crystals $B_v(\lambda)\otimes B_w(\infty)$ and obtain the following results.
\begin{theorem}\label{thm 4}
	For $\lambda \in P^+$ and $v,w\in W$, the following statements are equivalent.
	\begin{enumerate}
		\item $B_v(\lambda)\otimes B_w(\infty)$ is isomorphic to a disjoint union of Demazure crystals.
		\item $v_{\min}^{\lambda}\in \langle\{s_i\,|\, l(s_i w)< l(w)\}\rangle$.
		\item $B_v(\lambda)\otimes B_w(\infty)$ is an extremal subset of $B(\lambda)\otimes B(\infty)$.
	\end{enumerate}
\end{theorem}   
\thmref{thm 1} together with \thmref{thm 4} is a complete generalization of the results proved by Anthony Joseph in \cite{Anthony}. We also describe the Demazure crystals that appear in the decomposition of $B_v(\lambda)\otimes B_w(\infty)$ in the following theorem if $B_v(\lambda)\otimes B_w(\infty)$ is isomorphic to a disjoint union of Demazure crystals.
\begin{theorem}
	Let $\lambda \in P^+,\,v,w\in W, $ and $s_{i_1}\cdots s_{i_k}$ be a reduced expression of $v_{\min}^\lambda$ in $W$. Assume that $v_{\min}^{\lambda}\in \langle\{s_i\,|\, l(s_i w)< l(w)\}\rangle$ then for all $b\in B_w(\infty)^\lambda$ there exist $u(b,v)\in W$ such that
	$$B_v(\lambda)\otimes B_w(\infty)\cong \coprod_{b\in B_w(\infty)^{\lambda}}B_{u(b,v)}(\infty;\lambda+wt(b)).$$ 
\end{theorem}
 
 The paper is organized as follows. In \secref{prelims}, we fix the notations and recall some basic definitions used in the paper. In \secref{previous results}, we recall the results proved in \cite{Anthony}. In \secref{decomposition for finite}, we prove our main results and provide some applications. Finally in \secref{decomposition for infinity}, we prove \thmref{main result 2} and explain our results with the help of an example.  

	\section{PRELIMINARIES}\label{prelims}
	Let $\mathfrak{g}$ be a symmetrizable Kac-Moody Lie algebra with Cartan subalgebra $\mathfrak{h}$. Denote $\bu(\mathfrak{g})$ (resp. $\bu_q(\mathfrak{g})$) the corresponding universal enveloping algebra (resp. quantized universal enveloping algebra) of $\mathfrak{g}$. Denote by $I = \{1,2,\cdots n\}$ the indexing set and $\Delta = \{\alpha_1,\cdots, \alpha_n\}$ (resp. $\{\alpha_1^{\vee}, \cdots, \alpha_n^{\vee}\}$) the set of simple roots (resp. simple coroots) of $\mathfrak{g}$. Let $R$  (resp.  $R^+, R^-$) denote the set of roots (resp. positive, negative roots) of $\mathfrak{g}$.
	\subsection{} Denote the Weyl group of $\mathfrak{g}$ by $W$ and let $(\langle \alpha_i, \alpha_j^\vee \rangle)_{i,j \in I}$ be a symmetrizable generalized Cartan matrix. Let $Q$ (resp. $Q^+, Q^-$) and $P$ (resp. $P^+$) denote the root lattice (resp. positive root lattice, negative root lattice) and the weight lattice (resp. dominant weights). Let us denote the elements of $Q^+$ (resp. $Q^-$) by $\alpha>0$ (resp. $\alpha<0$). 
	
	For each $i\in I$, let $s_i$ denote the simple reflection in $W$ given by $$s_i(\beta) = \beta -\langle\beta, \alpha_i^\vee \rangle \alpha_i.$$ Given $w\in W$, let $w=s_{i_1}s_{i_2}\cdots s_{i_k}: i_j \in I$ be a reduced expression of $w$ then $k$ is the length of $w$ and $l(w)=k$ is the length function on $W$. The following lemma recalls a combinatorial result about Weyl group $W$.
	\begin{lemma}\label{length}\cite[Proposition 5.7]{Humphreys}
		For $w\in W$ and $ \alpha_i \in \Delta$, then \begin{enumerate}
			\item $l(ws_{i})>l(w)$ if and only if $w(\alpha_i)>0$.
			\item $l(s_{i}w)>l(w)$ if and only if $w^{-1}(\alpha_i)>0$.
		\end{enumerate} \qed
	\end{lemma}
	\subsection{} In this subsection, we recall the definitions of crystals, morphism between two crystals, and character of a crystal.
	
	\begin{definition}
		A \textit{crystal }$B$ is a set together with maps $wt: B\rightarrow P, \, \varphi_i, \varepsilon_i: B \rightarrow \mathbb Z \, \cup \{-\infty\},$ and $e_i, f_i: B\rightarrow B\, \cup\, \{0\}, \forall i\in I $ satisfying the following properties.
	\begin{description}
		\item[C1] For all $b \in B$ and $ i \in I$, $\varphi_i(b)  = \varepsilon_i(b)+ \langle wt(b), \alpha_i^{\vee}\rangle$,
		\item[C2]  For all $b,b^{'} \in B$ and $ i \in I$, $e_i(b) = b^{'} \Longleftrightarrow b= f_i b^{'}$,
		\item[C3] For all $b, e_i(b)\in B$ and $i\in I$, $wt(e_i(b)) = wt(b)+\alpha_i, \quad \varepsilon_i(e_i (b)) = \varepsilon_i(b)-1$,
		\item[C4] For all $b\in B$ and $i\in I$ with $\varphi_i(b) = -\infty$, $e_i(b) = f_i(b) =0$.
	\end{description}
			
	\end{definition}
	A crystal is said to be upper (resp. lower) normal if for all $b \in B$ and $ i\in I$, we have $\varepsilon_i(b) = \max \{k\,|\, e_i^k(b)\neq 0\}$ (resp. $\varphi_i(b) = \max \{k\,|\, f_i^k(b)\neq 0\}$). A \textit{normal crystal} is both upper and lower normal. We are mainly interested in upper normal and normal crystals. 
	
   Let $B_1$ and $B_2$ be crystals, a morphism $\psi$ is a map $\psi : B_1 \coprod\{0\} \rightarrow B_2 \coprod \{0\}$ that satisfies the following conditions.
   \begin{enumerate}
   	\item $\psi (0) =0$.
   	\item For all $i\in I$, $b\in B_1$ with $\psi(b)\in B_2$, we have $$wt(\psi(b)) = wt(b), \quad \varepsilon_i(\psi(b)) = \varepsilon_i(b),\quad \varphi_i(\psi(b)) = \varphi_i(b).$$
   	\item For all $i\in I$, $b\in B_1$ with $\psi(b)\in B_2,\, \psi(e_i(b))\in B_2$ we have $\psi(e_i(b)) = e_i(\psi(b)).$
   	\item For all $i\in I$, $b\in B_1$ with $ \psi(b)\in B_2, \, \psi(f_i(b))\in B_2$ we have $\psi(f_i(b)) = f_i(\psi(b)).$
   \end{enumerate} 
	A morphism $\psi$ is said to be strict if it commutes with both $e_i$ and $f_i$ for all $i\in I$. If $\psi$ is an injective morphism then $\psi $ is called an embedding.
	
	For a crystal $B$ and $\omega \in P$, define a set $B_\omega = \{b\in B|\, wt(b) = \omega\}$ and assume that $|B_\omega| < \infty$. Define the formal character $ch(B)$ by $$ch(B) = \sum_{\omega \in P}|B_{\omega}| e^\omega.$$   
	
	\subsection{} Let us recall the tensor product of two crystals. Let $B_1$ and $B_2$ be crystals and the set $B_1 \times B_2$ has a crystal structure denoted by $B_1 \otimes B_2$ defined in \cite{K1demazure}.
	\begin{definition}
	For $i\in I$ and $b_1 \otimes b_2 \in B_1 \otimes B_2,$ where $b_1 \in B_1$ and $b_2 \in B_2$, define the maps $wt, e_i, f_i$ as follows $$\begin{array}{ll}
			wt(b_1 \otimes b_2) = wt(b_1) +wt(b_2);  \\
			e_i(b_1 \otimes b_2) = \begin{cases}
				e_i(b_1)\otimes b_2 \quad \text{if }\varphi_i(b_1)\geq \varepsilon_i(b_2),\\
				b_1 \otimes e_i(b_2) \quad \text{if } \varphi_i(b_1)<\varepsilon_i(b_2),
			\end{cases} \\     
			f_i(b_1 \otimes b_2) = \begin{cases}
				f_i(b_1)\otimes b_2 \quad \text{if }\varphi_i(b_1)>\varepsilon_i(b_2),\\
				b_1 \otimes f_i(b_2) \quad \text{if } \varphi_i(b_1)\leq \varepsilon_i(b_2).
			\end{cases}
		\end{array} $$
	\end{definition}
	Then for each $i\in I$, it follows that $$\begin{array}{ll}
		\varepsilon_i(b_1 \otimes b_2) = \varepsilon_i(b_1)+\max\{0,\varepsilon_i(b_2)-\varphi_i(b_1)\},  \\
		\varphi_i(b_1 \otimes b_2) = \varphi_i(b_2)+\max\{0,\varphi_i(b_1)-\varepsilon_i(b_2)\},  \\
	\end{array}$$
	\subsection{} Now we recall the definition of Demazure crystals and some well known results proved in \cite{K1demazure, L2, Anthony}. Let $ B$ be a crystal. For $i \in I$ and $S\subseteq B$, define
	  $$T_iS = \{f_i^kb\, | \,b\in S, k\geq 0\} \backslash \{0\}\subseteq B.$$ For $\lambda \in P^+$, let $V(\lambda)$ be the irreducible highest weight $\bu_q(\mathfrak{g})$-module generated by $v_{\lambda}$ with highest weight $\lambda$. Let $B(\lambda)$ denotes its crystal basis with highest weight $b_{\lambda}$. Similarly, denote by $B(\infty)$, the crystal associated with $\bu_q^- (\mathfrak{g})$. The crystal $B(\infty)$ is a an upper normal highest weight crystal generated by $b_{\infty}$ of highest weight $0$. 
	  \begin{remark}\label{fib is nonzero}
	  	For all $i\in I$ and $b\in B(\infty)$ we have $f_i(b) \in B(\infty)$, that is $f_i(b)\neq 0$.
	  \end{remark}

	 For $w\in W$ with a reduced expression $ s_{i_1}\cdots s_{i_k}$ and $S\subseteq B$, define $$T_w S = T_{i_1}\cdots T_{i_k} S$$ and we know from \cite{K1demazure, L2} that the following subsets $$B_w(\lambda) = T_{w}\{b_{\lambda}\}, \quad B_w(\infty) = T_{w}\{b_{\infty}\}$$ of $B(\lambda)$ and $B(\infty)$ respectively are independent of the choice of a reduced expression of $w$.
	\begin{definition}
		For $\lambda \in P^+ $ and $w\in W$, the subset $B_w(\lambda)$ of $B(\lambda)$ is called the \textit{Demazure crystal}.
	\end{definition}
	Similarly, the subset $B_w(\infty)$ of $B(\infty)$ is called the limiting \textit{Demazure crystal}. For $w=e \in W$ and $\lambda \in P^+$, we have $B_e(\lambda) = \{b_\lambda\}$.
	Define an element $b_{w\lambda} \in B_w(\lambda)\subseteq B(\lambda)$ of weight $w\lambda$ by \begin{equation}\label{defn of bwlambda}
		b_{w\lambda} = f_{i_1}^{\langle s_{i_2}\cdots s_{i_n}\lambda, \alpha_{i_1}^\vee\rangle} \cdots f_{i_n}^{\langle\lambda, \alpha_{i_n}^\vee\rangle} b_\lambda
	\end{equation}
	The Demazure crystal $B_w(\lambda)$ is not necessarily stable under $f_i$ where $i\in I$ but it is stable under $e_i$ i.e. $e_i(B_w(\lambda))\subset B_w(\lambda)\cup \{0\}, \forall i \in I$. Therefore, a Demazure crystal is not necessarily a crystal. Also, note that the cardinality of $B_w(\lambda)$ is finite and $B_w(\infty)$ is infinite. 
	\begin{remark}\cite[Proposition 3.2.4]{K1demazure}\label{demazure subset}
		For $\lambda \in P^{+}\cup \{\infty\}$ and $w,w^{'} \in W$ such that $w\leq w^{'}$ in Bruhat order then $B_w(\lambda)\subseteq B_{w^{'}}(\lambda)$.
	\end{remark}
	\subsection{} Let us introduce the notion of extremal subsets. For $i\in I,\, \lambda \in P^+ \cup \{\infty\}$ and $b\in B(\lambda)$ with $e_i(b) =0$, the \textit{$i$-string} for $b$ is the set $S = \{f_i^k(b);  k \geq 0\}\backslash \{0\}$, where $b$ is called the highest weight vector of $S$. 
	
	In \cite{Assaf}, the notion of extremal subsets is defined for subsets of $B(\lambda)$. Now we extend the notion of extremal subsets for subsets of $B(\infty)$ and provide the following definition.
	\begin{definition}
		For $\lambda \in P^+ \cup \{\infty\}$ a non-empty subset $X$ of $B(\lambda)$ is said to be extremal if for any $i$-string $S$ of $B(\lambda)$ with highest weight vector $b$, we have $S \, \cap X$ is either $\emptyset, S,$ or $\{b\}$.
	\end{definition}
	Similarly, it is easy to define the notion of extremal subsets of $B(\lambda)\otimes B(\mu)$ where $\lambda \in P^+$ and $\mu \in P^+ \cup \{\infty\}$.
	\begin{remark}\label{union are extremal}
		It is easy to see that unions of extremal subsets are extremal.
	\end{remark}
	\begin{remark}\label{extremal subset}
		Using \cite[Proposition 3.3.5]{K1demazure}, we see that Demazure crystals are extremal subsets but not all extremal subsets are Demazure crystals.
	\end{remark}
	Now we recall the following result proved in \cite[Proposition 3.2.3]{K1demazure}.
	\begin{lemma}\label{Ti invariant}
		For $\lambda\in P^+ \cup \{\infty\}, i\in I$ and $w\in W$, we have $$T_iB_w(\lambda) = \begin{cases}
			B_w(\lambda)  & \text{if } l(s_iw)<l(w),\\
			B_{s_iw}(\lambda) & \text{if }l(s_iw)>l(w).
		\end{cases}$$ \qed
	\end{lemma}
	For $\lambda\in P^+$, define $W_{\lambda}= \{w\in W\,| \,w\lambda = \lambda\} = \langle\{s_i \,|\,s_i\lambda=\lambda, i\in I\}\rangle$, the stabilizer of $\lambda$ in $W$. 
	For $w\in W$, it follows from \cite{Humphreys} that there exist a unique element of smallest length in the coset $wW_{\lambda}$ and denote this minimal length representative by $w_{\min}^{\lambda}$.
	\begin{remark}\label{maximal length}
		In the case of symmetrizable Kac-Moody Lie algebras, it is not necessary that the maximal length element exist in the coset $wW_{\lambda}$. Although, it is true in the case of finite-dimensional Lie algebras.
	\end{remark}
	\section{Tensor decomposition of $B_e(\lambda)\otimes B_w(\mu)$}\label{previous results}
	In this section, we recall the results proved in \cite{K1demazure, L2, Anthony} for the decomposition of tensor product of highest weight crystals $B(\lambda)\otimes B(\mu)$ and Demazure crystals $B_e(\lambda)\otimes B_w(\mu)$.
	\subsection{} First we introduce the notion of primitive element given in \cite{Anthony}. Let $B$ be a crystal, an element $b\in B$ is said to be primitive if $e_i(b) =0,$ for all $i\in I$. Moreover, every primitive element of $B(\lambda)\otimes B(\mu)$ is of the form $b_{\lambda}\otimes b$ for some $b\in B(\mu)$.  Denote by $B(\mu)^{\lambda}$, the set of all $b\in B(\mu)$ such that $b_{\lambda}\otimes b$ is primitive in  $B(\lambda) \otimes B(\mu).$ From \cite{L1, L2}, it follows that for $\lambda, \mu \in P^+$, 
	 \begin{equation}\label{decom of highest}
		B(\lambda)\otimes B(\mu)= \coprod_{b\in B(\mu)^{\lambda}} \mathcal{F}(b_\lambda \otimes b)
	\end{equation}
	where $\mathcal{F}(b_\lambda \otimes b)$ denotes the connected component of $B(\lambda)\otimes B(\mu)$ containing $b_\lambda \otimes b$ and it obtains the crystal structure by restricting that of $B(\lambda)\otimes B(\mu)$. Moreover, as crystals
	\begin{equation}\label{F iso to highest}
		\mathcal{F}(b_\lambda \otimes b) \cong B(\lambda+wt(b)).
	\end{equation}
\begin{remark}\label{decom of lambda and infty}
	Taking $\mu \rightarrow \infty$, \eqref{decom of highest} also holds for $B(\lambda)\otimes B(\infty)$ and it is a disjoint union of highest weight crystals isomorphic to 
	$B(\infty)$, a crystal generated by $b_{\infty}$ of highest weight $\lambda+wt(b)$, where $b \in B(\infty)^\lambda$. 
\end{remark} 
For $b\in B(\infty)^{\lambda}$, denote this copy of $B(\infty)$ by $B(\infty; \lambda+wt(b))$. Using \eqref{decom of highest}, we have the following result. 
	\begin{theorem}\cite[Theorem 2.11]{Anthony}\label{identity result}
		For $\lambda \in P^+,\, \mu \in P^+\cup \{\infty\},$ and $w\in W$, then $B_e(\lambda)\otimes B_w(\mu)$ is isomorphic to a disjoint union of Demazure crystals. 
		Moreover, for all $b\in B(\mu)^{\lambda}\cap B_w(\mu)$ there exist $y_{w,b}^{\lambda}\in W$ such that  $$B_e(\lambda)\otimes B_w(\mu) = \coprod_{b\in B(\mu)^{\lambda}\cap B_w(\mu)}T_{y_{w,b}^{\lambda}}(b_{\lambda}\otimes b)$$ \qed
	\end{theorem} 
In \thmref{identity result}, we have from \cite{Anthony} that $B_w(\mu)^{\lambda} = B(\mu)^{\lambda}\cap B_w(\mu)$ is a finite set and $$T_{y_{w,b}^\lambda} (b_{\lambda}\otimes b)\cong \begin{cases}
	B_{y_{w,b}^\lambda}(\lambda+wt(b)), \text{ if } \mu <\infty,\\
	B_{y_{w,b}^\lambda}(\infty;\lambda+wt(b)), \text{ if } \mu =\infty.
\end{cases} $$	
In \cite{Anthony}, the criterion to compute all $b\in B_w(\mu)^\lambda$ and the corresponding $y_{w,b}^\lambda \in W$ is given.
  \subsection{} Now we define the character of the Demazure crystal $B_w(\mu)$ and state some useful applications of \thmref{identity result}. For $i\in I$, define a linear operator $\Delta_i$ on $\mathbb{Z}P$ by 
$$\Delta_i e^\mu = \frac{e^\mu - e^{s_i \mu -\alpha_i}}{1-e^{-\alpha_i}}.$$  
The operator $\Delta_i$ is called the \textbf{Demazure operator} associated with $i$. Also, we have $\Delta_i^2 = \Delta_i,\, \forall i \in I$. Given a reduced decomposition $s_{i_1}\cdots s_{i_k}$ of $w$ in $W$ and $\mu \in P^+$, Set $\Delta_w = \Delta_{i_1}\cdots \Delta_{i_k}$ then it follows from \cite{K1demazure, L2} that $$\ch B_w(\mu) = \Delta_w e^\mu.$$ Also, $D_w$ is independent of the choice of a reduced decomposition of $w$. For $\lambda, \mu \in P^+$ and $w\in W$, we have $$\ch (B_e(\lambda)\otimes B_w(\mu)) = e^\lambda (\Delta_w e^\mu).$$  Using \thmref{identity result}, we obtain the following corollaries.
\begin{corollary}\label{cor 1}
	For $\lambda, \mu \in P^+$ and $w\in W$, then $e^\lambda(\Delta_w e^\mu)$ is a linear combination of some Demazure characters with nonnegative integer coefficients.
\end{corollary}
\begin{corollary}\label{cor 2}
	For $\lambda, \mu \in P^+$, $w, y\in W$, we have $$\ch T_y (B_e(\lambda)\otimes B_w(\mu)) = \Delta_y(e^\lambda(\Delta_w e^\mu)).$$
\end{corollary}

	\section{Tensor decomposition of $B_v(\lambda)\otimes B_w(\mu)$}\label{decomposition for finite}
	In this section, we prove results for the decomposition of the tensor product of Demazure crystals $B_v(\lambda)\otimes B_w(\mu)$ which are generalization of \thmref{identity result}. Our results are also the generalization of the results proved in \cite{Assaf,Knouo} for finite-dimensional Lie algebras to symmetrizable Kac-Moody Lie algebras. 
	\subsection{} We provide necessary and sufficient conditions for the decomposition of the tensor product of Demazure crystals $B_v(\lambda)\otimes B_w(\mu)$ in the following theorem.
	\begin{theorem}\label{main thm}
		For $\lambda,\mu \in P^+$ and $v,w\in W$, the following statements are equivalent.
		\begin{enumerate}
			\item $B_v(\lambda)\otimes B_w(\mu)$ is isomorphic to a disjoint union of Demazure crystals.
			\item $v_{\min}^{\lambda}\in \langle\{s_i\,|\, \langle w\mu,\alpha_i^\vee\rangle \leq 0\}\rangle$.
			\item $B_v(\lambda)\otimes B_w(\mu)$ is an extremal subset of $B(\lambda)\otimes B(\mu)$.
		\end{enumerate}
	\end{theorem}
	\begin{remark}
	Using \remref{maximal length}, we see that the necessary and sufficient condition given in \cite{Knouo} for the decomposition of the tensor product of Demazure crystals, does not work in the case of symmetrizable Kac-Moody Lie algebras.
	\end{remark}
The rest of this section covers the proof of \thmref{main thm}.	
	\begin{lemma}\label{i implies ii}
		Let $\lambda, \mu \in P^+$ and $v,w \in W$. Assume that  $B_v(\lambda)\otimes B_w(\mu)$ is an extremal subset of $B(\lambda)\otimes B(\mu)$ then $v_{\min}^\lambda \in \langle\{s_i\,|\,\langle w\mu,\alpha_i^\vee\rangle \leq 0\}\rangle$.
	\end{lemma}
	\begin{proof}
	Let $v_{\min}^{\lambda} = s_{i_1}\cdots s_{i_k}$ be a reduced expression of $v_{\min}^\lambda$ in $ W$. Assume that $v_{\min}^{\lambda}\notin \langle\{s_i\,|\,\langle w\mu,\alpha_i^\vee\rangle \leq 0\}\rangle$. Then there exists $i_j$ such that $\langle w\mu,\alpha_{i_j}^\vee\rangle \,> 0$ which implies that 
	$$\varepsilon_{i_j}(b_{w\mu}) =0, \quad \varphi_{i_j}(b_{w\mu}) = \langle w\mu,\alpha_{i_j}^\vee\rangle \,> 0.$$ 
	where $b_{w\mu}$ is defined in \eqref{defn of bwlambda}. Then we have
	 \begin{equation}\label{eij is zero}
		e_{i_j}(b_{w\mu}) =0, \quad f_{i_j}(b_{w\mu}) \neq 0. 
	\end{equation} Moreover, $wt(f_{i_j}(b_{w\mu}))=w\mu-\alpha_{i_j}<w\mu$ but $w\mu$ is the lowest weight in $B_w(\mu)$ . Therefore, we have $f_{i_j}(b_{w\mu})\notin B_w(\mu)\coprod \{0\}$.\\
	
 \textbf{ Claim:}	$ f_{i_j}(b_{s_{i_{j+1}}\cdots s_{i_k}\lambda}) \neq 0 \text{ and } f_{i_j}(b_{s_{i_{j+1}}\cdots s_{i_k}\lambda})\in B_v(\lambda).$\\
		
		(a) Suppose $\langle s_{i_{j+1}}\cdots s_{i_k} \lambda\,,\alpha_{i_j}^{\vee}\rangle < 0$, then
		$\langle  \lambda\,,(s_{i_{j+1}}\cdots s_{i_k})^{-1}\alpha_{i_j}^{\vee}\rangle < 0$. We have $(s_{i_{j+1}}\cdots s_{i_k})^{-1}\alpha_{i_j}^{\vee}$ is a negative coroot as $\lambda$ is a dominant integral weight. Using \lemref{length} we have $$l(s_{i_j}s_{i_{j+1}}\cdots s_{i_k})<l(s_{i_{j+1}} \cdots s_{i_k})$$ which is a contradiction as $s_{i_1}\cdots s_{i_k}$ is a reduced expression of $v_{\min}^{\lambda}$.
		 
		(b) Suppose $\langle s_{i_{j+1}}\cdots s_{i_k} \lambda\,,\alpha_{i_j}^{\vee}\rangle = 0$, then $ s_{i_{j}}(s_{i_{j+1}}\cdots s_{i_k} \lambda)\, = s_{i_{j+1}}\cdots s_{i_k}\lambda$. We have $$vW_\lambda = v_{\min}^{\lambda}W_\lambda = s_{i_1}\cdots s_{i_{j-1}}s_{i_{j+1}}\cdots s_{i_k}W_\lambda \implies s_{i_1}\cdots s_{i_{j-1}}s_{i_{j+1}}\cdots s_{i_k} \in vW_{\lambda}$$ which is a contradiction as $v_{\min}^{\lambda}$ is a minimal length representative of $vW_{\lambda}$.
		
		(c) Suppose $\langle s_{i_{j+1}}\cdots s_{i_k} \lambda\,,\alpha_{i_j}^{\vee}\rangle > 0$ then $\varepsilon_{i_j}(b_{s_{i_{j+1}}\cdots s_{i_k}\lambda}) =0$ which implies that $$e_{i_j}(b_{s_{i_{j+1}}\cdots s_{i_k}\lambda}) =0,\, \text{ and }f_{i_j}(b_{s_{i_{j+1}}\cdots s_{i_k}\lambda}) \neq 0.$$
		
		Now we prove that $f_{i_j}(b_{s_{i_{j+1}}\cdots s_{i_k}\lambda})\in B_v(\lambda)$. We have $ f_{i_j}(b_{s_{i_{j+1}}\cdots s_{i_k}\lambda})\in B_{s_{i_j}\cdots s_{i_k}}(\lambda)$ and $s_{i_j}\cdots s_{i_k}$ is a subexpression of a reduced expression $s_{i_1}\cdots s_{i_k}= v_{\min}^\lambda$. Then $s_{i_j}\cdots s_{i_k} \leq v_{\min}^\lambda$ in Bruhat order. Using \remref{demazure subset}, we have $$B_{s_{i_j}\cdots s_{i_k}}(\lambda)\subseteq B_{v_{\min}^\lambda}(\lambda) = B_v(\lambda).$$ Hence, we have the claim.

		Using the claim proved above and \eqref{eij is zero}, we have 
	   $$f_{i_j}(b_{s_{i_{j+1}}\cdots s_{i_k}\lambda})\neq 0, \quad e_{i_j}(b_{w\mu})=0$$ 
	   which implies that $\varphi_{i_j}(b_{s_{i_{j+1}}\cdots s_{i_k}\lambda})>0=\varepsilon_{i_j}(b_{w\mu})$. By the tensor product rule and the claim, we get
	   $$f_{i_j}(b_{s_{i_{j+1}}\cdots s_{i_k}\lambda}\, \otimes b_{w\mu}) = f_{i_j}b_{s_{i_{j+1}}\cdots s_{i_k}\lambda}\, \otimes b_{w\mu}\in B_v(\lambda)\otimes B_w(\mu).$$
	   Let $\varphi_{i_j}(b_{s_{i_{j+1}}\cdots s_{i_k}\lambda}) =n$ then
	   $$f_{i_j}^{n+1}(b_{s_{i_{j+1}}\cdots s_{i_k}\lambda}\, \otimes b_{w\mu}) = f_{i_j}^n(b_{s_{i_{j+1}}\cdots s_{i_k}\lambda})\otimes f_{i_j}(b_{w\mu}) \neq 0$$ 
	   but $f_{i_j}(b_{w\mu})\notin B_w(\mu)$ which implies that 
	   $$f_{i_j}^{n+1}(b_{s_{i_{j+1}}\cdots s_{i_k}\lambda}\, \otimes b_{w\mu})\notin B_v(\lambda)\otimes B_w(\mu).$$
	   Hence, $B_v(\lambda)\otimes B_w(\mu)$ does not contain the $i_j$-string which implies that it is not an extremal subset of $B(\lambda)\otimes B(\mu)$ which is a contradiction. Therefore, we have $v_{\min}^{\lambda}\in \langle\{s_i\,|\,\langle w\mu,\alpha_i^\vee\rangle \leq 0\}\rangle$.  
	\end{proof}
	\subsection{} Let $v = s_{i_1}\cdots s_{i_k}$ be a reduced expression of $v$ in $W$. Denote \begin{equation}\label{Tv of blambda}
	T_{v,\lambda,w,\mu} = T_v(b_{\lambda}\otimes B_w(\mu)) = T_{i_1}\cdots T_{i_k}(b_{\lambda}\otimes B_w(\mu)).
	\end{equation} 
	\begin{lemma}\label{lem2}
		Let $\lambda,\mu \in P^+$ and  $v,w\in W$ such that $s_{i_1} \cdots s_{i_k}$ is a reduced expression of $v$ in $W$. Then the following statements hold:
		\begin{enumerate}
			\item $B_v(\lambda)\otimes B_w(\mu) \subseteq  T_{v,\lambda,w,\mu}$.
			\item $B_v(\lambda)\otimes B_w(\mu) = T_{v,\lambda,w,\mu}$ if $v\in \langle\{s_i\,|\,\langle w\mu,\alpha_i^\vee\rangle \leq 0\}\rangle$.
		\end{enumerate} 
	\end{lemma}
	\begin{proof}\textit{(1)}
	Let $ b_1 \otimes b_2 $ be a non-zero element in $B_v(\lambda)\otimes B_w(\mu)$. Using \eqref{Tv of blambda}, we need to prove that $b_1 \otimes b_2 \in T_{i_1}\cdots T_{i_k}(b_{\lambda}\otimes B_w(\mu))$. By the definition of crystals, it is sufficient to prove that 
		$$e_{i_k}^{\varepsilon_{i_k}(b_1\otimes b_2)}\cdots e_{i_1}^{\varepsilon_{i_1}(b_1\otimes b_2)}(b_1 \otimes b_2)\in b_{\lambda}\otimes B_w(\mu).$$
		 Since $b_1 \in B_v(\lambda) = T_{i_1}\cdots T_{i_k}\{b_{\lambda}\}$, then $b_1 = f_{i_1}^{c_1}\cdots f_{i_k}^{c_k}(b_{\lambda})$ where $c_j := \varepsilon_{i_j}(e_{i_{j-1}}^{c_{j-1}}\cdots e_{i_1}^{c_1}(b_1))$ defined inductively for all $j=1,2,\cdots,k$. Then there exist $s_1\geq 0$ such that 
	 $$e_{i_1}^{\varepsilon_{i_1}(b_1\otimes b_2)}(f_{i_1}^{c_1}\cdots f_{i_k}^{c_k}b_{\lambda}\otimes b_2) = f_{i_2}^{c_2}\cdots f_{i_k}^{c_k}b_{\lambda}\otimes e_{i_1}^{s_1}b_2.$$ 
		 Continue like this, we get $s_j\geq 0$ for all $ j\in \{2,3,\cdots,k\}$ such that
	 $$e_{i_k}^{\varepsilon_{i_k}(b_1\otimes b_2)} \cdots e_{i_1}^{\varepsilon_{i_1}(b_1\otimes b_2)}(f_{i_1}^{c_1}\cdots f_{i_k}^{c_k}b_{\lambda}\otimes b_2) = b_{\lambda}\otimes e_{i_k}^{s_k}\cdots e_{i_1}^{s_1}b_2.$$ 
		Since $B_w(\mu)$ is stable under $e_{i}$ for all $i\in I$  and $e_{i_k}^{\varepsilon_{i_k}(b_1\otimes b_2)}\cdots e_{i_1}^{\varepsilon_{i_1}(b_1\otimes b_2)}(b_1 \otimes b_2)\neq 0$ which implies that $e_{i_k}^{s_k}\cdots e_{i_1}^{s_1}b_2 \in B_w(\mu)$. Hence, $$ e_{i_k}^{\varepsilon_{i_k}(b_1\otimes b_2)}\cdots e_{i_1}^{\varepsilon_{i_1}(b_1\otimes b_2)}(b_1 \otimes b_2)\in b_{\lambda}\otimes B_w(\mu).$$
 \\
\textit{(2)} For $b\in B_w(\mu)$, we need to prove that $T_{i_j}(b)\in B_w(\mu), \forall j \in \{1,2,\cdots, k\}.$ We are given that $v\in \langle\{s_i\,|\,\langle w\mu,\alpha_i^\vee\rangle \,\leq 0\}\rangle$. Therefore, $\langle w\mu,\alpha_{i_j}^\vee\rangle \leq 0, \, \forall j \in \{1,2,\cdots, k\}.$\\

  \textit{Case I:} If $\langle w\mu,\alpha_{i_j}^\vee\rangle <0$, then $\langle \mu, w^{-1}\alpha_{i_j}^\vee\rangle < 0$. We get that $w^{-1}\alpha_{i_j}^\vee$ is a negative coroot as $\mu$ is a dominant integral weight. Using \lemref{length} and \lemref{Ti invariant}, we have $$l(s_{i_j}w)<l(w) \implies T_{i_j}B_w(\mu) = B_w(\mu)$$
   
   \textit{Case II:} If $\langle w\mu,\alpha_{i_j}^\vee\rangle =0$, then $s_{i_j}w\mu = w\mu$ and $s_{i_j}wW_{\mu} = wW_{\mu}$. This implies that $$T_{i_j}B_w(\mu) = B_{s_{i_j}w}(\mu) = B_w(\mu).$$ Hence, for all $j\in \{1,2,\cdots,k\}$ we have 
\begin{equation}\label{Tij(b)}
			T_{i_j}(b)\in B_w(\mu).
		\end{equation} 
		For all $b\in B_w(\mu)$ and non-negative integers $c_1,\cdots, c_k$, we need to prove that
  \begin{equation}\label{fi1...fik}
  	f_{i_1}^{c_1}\cdots f_{i_k}^{c_k}(b_{\lambda}\otimes b)\in B_v(\lambda)\otimes B_w(\mu) \text{ if } f_{i_1}^{c_1}\cdots f_{i_k}^{c_k}(b_{\lambda}\otimes b)\neq 0.
  \end{equation}
		Suppose $f_{i_1}^{c_1}\cdots f_{i_k}^{c_k}(b_{\lambda}\otimes b) \neq 0$, then there exist non-negative integers $a_k, d_k$ such that 
	$$f_{i_k}^{c_k}(b_{\lambda}\otimes b) = f_{i_k}^{a_k}(b_{\lambda})\otimes f_{i_k}^{d_k}(b).$$ 
	    Using \eqref{Tij(b)}, we have $f_{i_k}^{d_{k}}(b)\in B_w(\mu)$. Continue like this we get there exist non-negative integers $a_1, \cdots, a_k$ and $d_1, \cdots, d_k$ such that
  $$f_{i_1}^{c_1}\cdots f_{i_k}^{c_k}(b_{\lambda}\otimes b) = f_{i_1}^{a_1}\cdots f_{i_k}^{a_k}(b_{\lambda})\otimes f_{i_1}^{d_1}\cdots f_{i_k}^{d_k}(b).$$ 
	    Since $s_{i_1}\cdots s_{i_k}$ is a reduced expression of $v$ in $W$ and $f_{i_1}^{c_1}\cdots f_{i_k}^{c_k}(b_{\lambda}\otimes b)\neq 0$ then $f_{i_1}^{a_1}\cdots f_{i_k}^{a_k}(b_{\lambda})\in B_v(\lambda)$. Using \eqref{Tij(b)}, we have $f_{i_1}^{d_1}\cdots f_{i_k}^{d_k}(b) \in B_w(\mu)$. Hence, 
	    $$T_{v,\lambda,w,\mu} = T_{i_1}\cdots T_{i_k}(b_{\lambda}\otimes B_w(\mu)) \subseteq B_v(\lambda)\otimes B_w(\mu).$$  
	\end{proof}
		\subsection{}\label{uj subsection} From \thmref{identity result}, we have 
	$$B_e(\lambda)\otimes B_w(\mu) = \coprod_{b\in B_w(\mu)^{\lambda}}T_{y_{w,b}^{\lambda}}(b_{\lambda}\otimes b)$$
	For $v\in W$, let $s_{i_1}\cdots s_{i_k}$ be a reduced expression of $v_{\min}^{\lambda}$ in $W$. Define $u_j \in W$ inductively for $j=1,2,\cdots, k$. 
	
	Set 
	$$u_k := \begin{cases}
		s_{i_k}y_{w,b}^{\lambda} & \text{ if } l(s_{i_k}y_{w,b}^{\lambda}) > l(y_{w,b}^{\lambda}),\\
		y_{w,b}^{\lambda} & \text{ if } l(s_{i_k}y_{w,b}^{\lambda})<l(y_{w,b}^{\lambda}).
	\end{cases}$$ 
	Suppose $u_j \in W$ is defined for some $j\in \{2,\cdots,k\}$, then define $$u_{j-1} := \begin{cases}
		s_{i_{j-1}}u_j & \text{ if } l(s_{i_{j-1}}u_j) > l(u_j),\\
		u_j & \text{ if } l(s_{i_{j-1}}u_j)<l(u_j).
	\end{cases}$$
	For $b\in B_{w}(\mu)^{\lambda}$, define $u(b,v):= u_1$. The definition of $u(b,v)$ depends on the choice of a reduced expression of $v_{\min}^{\lambda}$.
	The following theorem proves that $B_v(\lambda)\otimes B_w(\mu)$ is isomorphic to the disjoint union of Demazure crystals if $v_{\min}^{\lambda}\in \langle \{s_i\, | \,\langle w\mu,\alpha_{i}^{\vee}\rangle\leq 0\} \rangle$ and explicitly mention the Demazure crystals that appear in the disjoint union.  
	\begin{theorem}\label{disjoint union}
		Let $\lambda, \mu \in P^+,\,v,w\in W, $ and $s_{i_1}\cdots s_{i_k}$ be a reduced expression of $v_{\min}^\lambda$ in $W$. Assume that $v_{\min}^{\lambda}\in \langle \{s_i\, | \langle w\mu,\alpha_{i}^{\vee}\rangle\leq 0\} \rangle$, then 
		$$B_v(\lambda)\otimes B_w(\mu)\cong \coprod_{b\in B_w(\mu)^{\lambda}}B_{u(b,v)}(\lambda+wt(b)).$$ 
	\end{theorem}
	\begin{proof} We are given that 
		\begin{equation} \label{lessthan}
			v_{\min}^{\lambda}\in \langle \{s_i\,|\, \langle w\mu ,\alpha_i^{\vee}\rangle\leq 0\}\rangle.\end{equation} 
			From \eqref{decom of highest}, we have 
			\begin{equation}\label{connected components}
				B_v(\lambda)\otimes B_w(\mu) = \coprod_{b\in B_w(\mu)^\lambda} \mathcal{F}(b_{\lambda}\otimes b,v)
			\end{equation} 
			where $\mathcal{F}(b_{\lambda}\otimes b,v)$ denotes the connected component of $B_v(\lambda)\otimes B_w(\mu)$ containing $b_\lambda \otimes b$.	
		Using \eqref{lessthan}, \lemref{lem2}, and $vW_{\lambda} = v_{\min}^{\lambda}W_{\lambda}$, we have
		\begin{equation}\label{tensor equal t}
			B_v(\lambda)\otimes B_w(\mu) = B_{v_{\min}^{\lambda}}(\lambda)\otimes B_{w}(\mu)
			= T_{v_{\min}^{\lambda},\lambda,w,\mu} = T_{i_1}\cdots T_{i_k}(b_{\lambda} \otimes B_{w}(\mu))
		\end{equation} 
		Using \eqref{tensor equal t} and \thmref{identity result}, we have $$\begin{array}{ll}
			B_v(\lambda)\otimes B_w(\mu)   &= T_{i_1}\cdots T_{i_k}(\coprod_{b\in B_w(\mu)^{\lambda}} T_{y_{w,b}^{\lambda}}(b_{\lambda}\otimes b))\\
			& = \bigcup_{b\in B_w(\mu)^{\lambda}}(T_{i_1}\cdots T_{i_k}T_{y_{w,b}^{\lambda}}(b_{\lambda}\otimes b))
		\end{array}$$
		The subset $T_{i_1}\cdots T_{i_k}(T_{y_{w,b}^{\lambda}}(b_{\lambda}\otimes b))$ of $B_v(\lambda)\otimes B_w(\mu)$ is connected and contains $b_{\lambda}\otimes b$ which implies that $T_{i_1}\cdots T_{i_k}(T_{y_{w,b}^{\lambda}}(b_{\lambda}\otimes b))\subseteq \mathcal{F}( b_\lambda \otimes b,v).$ Using \eqref{connected components}, we have  $T_{i_1}\cdots T_{i_k}(T_{y_{w,b}^{\lambda}}(b_{\lambda}\otimes b)) =  \mathcal{F}(b_\lambda \otimes b,v)$ and 
		$$B_v(\lambda)\otimes B_w(\mu)= \coprod_{b\in B_w(\mu)^{\lambda}}T_{i_1}\cdots T_{i_k}(T_{y_{w,b}^{\lambda}}(b_{\lambda}\otimes b)).$$ 
	Using \thmref{identity result}, we get that $T_{y_{w,b}^{\lambda}}(b_{\lambda}\otimes b)\cong B_{y_{w,b}^{\lambda}}(\lambda+wt(b))$ which implies that $$T_{i_1}\cdots T_{i_k}(T_{y_{w,b}^{\lambda}}(b_{\lambda}\otimes b)) \cong T_{i_1}\cdots T_{i_k}(B_{y_{w,b}^{\lambda}}(\lambda+wt( b))).$$
		Using \lemref{Ti invariant}, we have $$T_{i_k}B_{y_{w,b}^{\lambda}}(\lambda+wt( b)) = \begin{cases}
			B_{s_{i_k}y_{w,b}^{\lambda}}(\lambda+wt( b)) & \text{if } l(s_{i_k}y_{w,b}^{\lambda})>l(y_{w,b}^{\lambda})\\
			B_{y_{w,b}^{\lambda}}(\lambda+wt( b))) & \text{if } l(s_{i_k}y_{w,b}^{\lambda})< l(y_{w,b}^{\lambda})
		\end{cases}$$
		which is same as $B_{u_k}(\lambda+wt(b))$. By induction assume that $$T_{i_j}\cdots T_{i_k}B_{y_{w,b}^{\lambda}}(\lambda+wt( b))= B_{u_j}(\lambda+wt(b))$$ for some $j\in \{2,\cdots,k\}$. We get that $$\begin{array}{ll}
			T_{i_{j-1}}T_{i_j}\cdots T_{i_k} B_{y_{w,b}^{\lambda}}(\lambda+wt( b))   &= T_{i_{j-1}}B_{u_{j}}(\lambda+wt (b)) \\
			& = \begin{cases}
				B_{s_{i_{j-1}}u_j}(\lambda+wt(b)), & \text{if } l(s_{i_{j-1}}u_j)>l(u_j),\\
				B_{u_j}(\lambda+wt(b)), & \text{if } l(s_{i_{j-1}}u_{j})< l(u_j)
			\end{cases}\\
			& = B_{u_{j-1}}(\lambda+wt(b)).
		\end{array}$$
		Hence, we have $$T_{i_1}\cdots T_{i_k}B_{y_{w,b}^{\lambda}}(\lambda+wt(b)) = B_{u_1}(\lambda+wt(b))= B_{u(b,v)}(\lambda+wt(b)).$$
	\end{proof}
	\subsection{Proof of the \thmref{main thm}}\label{proof of main thm in sec 4}  
	$(1)\implies (3)$ We are given that $B_v(\lambda)\otimes B_w(\mu)$ is isomorphic to a disjoint union of Demazure crystals. Using \remref{extremal subset}, we see that each Demazure crystal is an extremal subset. Therefore, using \remref{union are extremal} we have $B_v(\lambda)\otimes B_w(\mu)$ is an extremal subset of $B(\lambda)\otimes B(\mu)$.
	
	$(3)\implies (2)$ Using \lemref{i implies ii}, we have $v_{\min}^{\lambda}\in \langle \{s_i\, | \langle w\mu,\alpha_{i}^{\vee}\rangle\leq 0\} \rangle$.

	$(2)\implies(1)$ Let $ s_{i_1}\cdots s_{i_k}$ be a reduced expression of $v_{\min}^\lambda$ in $W$. Using \thmref{disjoint union}, we get that $B_v(\lambda)\otimes B_w(\mu)$ is isomorphic to a disjoint union of Demazure crystals. \qed

\subsection{Applications}	  Let us look at a few applications of \thmref{main thm}. Using \lemref{lem2} and \corref{cor 2}, we obtain the following corollary which is a generalization of \corref{cor 2}.
	  \begin{corollary}\label{cor3}
	  		Let $\lambda,\mu \in P^+$ and  $v,w\in W$ such that $s_{i_1} \cdots s_{i_k}$ is a reduced expression of $v$ in $W$. Assume that $v_{\min}^{\lambda} \in \langle \{s_i\,| \langle  w\mu, \alpha_i^{\vee} \rangle\leq 0 \}\rangle$ then $$\ch (B_v(\lambda)\otimes B_w(\mu)) = \ch T_{v,\lambda,w,\mu} = \ch T_v(B_e(\lambda)\otimes B_w(\mu)) = \Delta_v(e^\lambda(\Delta_w)e^\mu).$$
	  \end{corollary} 
	  Using \corref{cor3} and \thmref{main thm}, we obtain the criterion when the product of two Demazure characters is a nonnegative linear combination of Demazure characters which is a generalization of \corref{cor 1}.
	  \begin{corollary}\label{cor 4}
	  	For $\lambda,\mu \in P^+$ and  $v,w\in W$, assume that $v_{\min}^{\lambda} \in \langle \{s_i\,| \langle  w\mu, \alpha_i^{\vee} \rangle\leq 0 \}\rangle$. Then $\Delta_v(e^\lambda(\Delta_w)e^\mu)$ is a linear combination of some Demazure characters with nonnegative integer coefficients.
	  \end{corollary}
 In particular, if $\mathfrak{g}$ is a simple finite-dimensional Lie algebra of type $A$. Then \corref{cor 4} provides a partial solution for the \textit{key positivity problem} that is when the product of two key polynomials is a linear combination of key polynomials with nonnegative integer coefficients. 
 
 Let $\nu \in P$, then there exist $u\in W$ and $\lambda_\nu \in P^+$ such that $u\lambda_{\nu} = \nu$. Set $u_\nu = u_{\min}^{\lambda_{\nu}}$ and define the \textbf{key polynomial} $\kappa_\nu$ to be $$\kappa_\nu = \ch B_{u_\nu}(\lambda_\nu).$$
 Using \thmref{disjoint union}, we obtain the following result.
 \begin{theorem}
 	For $\lambda,\mu \in P^+$ and  $v,w\in W$, assume that $v_{\min}^{\lambda} \in \langle \{s_i\,| \langle  w\mu, \alpha_i^{\vee} \rangle\leq 0 \}\rangle$. Then there exist a nonnegative integer $c_{v,\lambda,w,\mu}^\nu$ for each $\nu \in P$ such that $$\kappa_{v\lambda}\kappa_{w\mu} = \sum_{\nu \in P}c_{v,\lambda,w,\mu}^\nu \kappa_\nu.$$
 \end{theorem} 	
 \begin{proof}
 	The proof follows from \thmref{disjoint union} by taking character on both the sides.
 \end{proof}
 The following corollary is a particular case of \corref{cor 4}.
 \begin{corollary}
 	For $\lambda,\mu \in P^+$ and  $v,w\in W$, assume that either $v_{\min}^{\lambda} \in \langle \{s_i\,| \langle  w\mu, \alpha_i^{\vee} \rangle\leq 0 \}\rangle$ or $w_{\min}^\mu \in \langle \{s_i\,| \langle  v\lambda, \alpha_i^{\vee} \rangle\leq 0 \}\rangle $. Then the product of $\kappa_{v\lambda}$ and $\kappa_{w\mu}$ is a linear combination of some key polynomials with nonnegative integer coefficients.
 \end{corollary}  
 
	  \section{Tensor decomposition of $B_v(\lambda)\otimes B_w(\infty)$}\label{decomposition for infinity}
	  In this section, we study the decomposition of the tensor product of  Demazure crystals $B_v(\lambda)\otimes B_w(\infty)$ for symmetrizable Kac-Moody Lie algebra $\mathfrak{g}$. We extend \thmref{main thm} by including the Demazure crystal $B_w(\infty)$. \thmref{main thm} together with \thmref{main result 2} is a generalization of \thmref{identity result} for any $v,w\in W$. 
	 
	  \subsection{} We provide necessary and sufficient conditions for the decomposition of the tensor product of Demazure crystals $B_v(\lambda)\otimes B_w(\infty)$ in the following theorem.
	  
	  \begin{theorem}\label{main result 2}
	  	For $\lambda \in P^+$ and $v,w\in W$, the following statements are equivalent.
	  	\begin{enumerate}
	  		\item $B_v(\lambda)\otimes B_w(\infty)$ is isomorphic to a disjoint union of Demazure crystals.
	  		\item $v_{\min}^{\lambda}\in \langle\{s_i\,|\, l(s_i w)< l(w)\}\rangle$.
	  		\item $B_v(\lambda)\otimes B_w(\infty)$ is an extremal subset of $B(\lambda)\otimes B(\infty)$.
	  	\end{enumerate}
	  \end{theorem}   
	  The rest of this section covers the proof of \thmref{main result 2}.	
	  \begin{lemma}\label{A implies B}
	  	Let $\lambda \in P^+$ and $v,w \in W$. Assume that $B_v(\lambda)\otimes B_w(\infty)$ is an extremal subset of $B(\lambda)\otimes B(\infty)$ then $v_{\min}^{\lambda}\in \langle\{s_i\,|\, l(s_i w)< l(w)\}\rangle$.
	  \end{lemma}
	  \begin{proof} Let $v_{\min}^{\lambda} =s_{i_1}\cdots s_{i_k}$ be a reduced expression of $v_{\min}^{\lambda}$ in $W$. Assume that $v_{\min}^{\lambda}\notin \langle\{s_i\,|\, l(s_i w)< l(w)\}\rangle$. Then there exist $i_j$ such that 
	  	 $l(s_{i_j}w)> l(w)$ which implies that $s_{i_j}w\geq w$ in Bruhat order. Using \lemref{Ti invariant} and \remref{demazure subset} we have 
	  	$$T_{i_j}B_w(\infty) = B_{s_{i_j} w}(\infty)\supsetneq B_w(\infty).$$
	  	 Using \remref{fib is nonzero} and \remref{extremal subset} there exist an element $b\in B_w(\infty)$ such that 
	  	\begin{equation}\label{eijb is zero}
	  		f_{i_j}(b)\neq 0, \quad f_{i_j}(b)\notin B_w(\infty), \quad e_{i_j}(b) =0.
	  	\end{equation}
	  	Consider an element $b_{s_{i_{j+1}}\cdots s_{i_k}\lambda}$ in $B_v(\lambda)$. 
	  	
	   \textbf{ Claim:}	$ f_{i_j}(b_{s_{i_{j+1}}\cdots s_{i_k}\lambda}) \neq 0 \text{ and } f_{i_j}(b_{s_{i_{j+1}}\cdots s_{i_k}\lambda})\in B_v(\lambda).$\\
	 The proof of the above claim is exactly the same as given in the proof of \lemref{i implies ii}. \\
	 
	 Since $B_v(\lambda)\otimes B_w(\infty)$ is an extremal subset of $B(\lambda)\otimes B(\infty)$. Using the claim and \eqref{eijb is zero}, we have $\varphi_{i_j}(b_{s_{i_{j+1}}\cdots s_{i_k}}\lambda)>0=\varepsilon_{i_j}(b)$ which implies that
 $$f_{i_j}(b_{s_{i_{j+1}}\cdots s_{i_k}\lambda}\, \otimes b) = f_{i_j}b_{s_{i_{j+1}}\cdots s_{i_k}\lambda}\, \otimes b\in B_v(\lambda)\otimes B_w(\infty).$$
	  Let $\varphi_{i_j}(b_{s_{i_{j+1}}\cdots s_{i_k}\lambda}) =n$ then
	$$f_{i_j}^{n+1}(b_{s_{i_{j+1}}\cdots s_{i_k}\lambda}\, \otimes b) = f_{i_j}^n(b_{s_{i_{j+1}}\cdots s_{i_k}\lambda})\otimes f_{i_j}(b)\notin B_v(\lambda)\otimes B_w(\infty) $$ 
	as $f_{i_j}(b)\notin B_w(\infty)$ and $f_{i_j}^{n+1}(b_{s_{i_{j+1}}\cdots s_{i_k}\lambda}\, \otimes b) \neq 0$.
	 Hence, $B_v(\lambda)\otimes B_w(\infty)$ does not contain the $i_j$-string which implies that it is not an extremal subset of $B(\lambda)\otimes B(\infty)$ which is a contradiction. Therefore, we have $v_{\min}^{\lambda}\in \langle\{s_i\,|\, l(s_i w)< l(w)\}\rangle$.   	
	  \end{proof}
\subsection{} Let $v = s_{i_1}\cdots s_{i_k}$ be a reduced expression of $v$ in $W$. Denote \begin{equation}
	T_{v,\lambda,w} = T_v(b_{\lambda}\otimes B_w(\infty)) = T_{i_1}\cdots T_{i_k}(b_{\lambda}\otimes B_w(\infty)).
\end{equation} 
\begin{lemma}\label{lem infty}
	Let $\lambda \in P^+$ and  $v,w\in W$ such that $s_{i_1} \cdots s_{i_k}$ is a reduced expression of $v$ in $W$. Then the following statements hold:
	\begin{enumerate}
		\item $B_v(\lambda)\otimes B_w(\infty) \subseteq  T_{v,\lambda,w}$.
		\item $B_v(\lambda)\otimes B_w(\infty) = T_{v,\lambda,w}$ if $v\in \langle\{s_i\,|\, l(s_i w)< l(w)\}\rangle$.   	
	\end{enumerate} 
\end{lemma}
\begin{proof}
	\textit{(1)} The proof of this part is same as the proof of part \textit{(1)} of \lemref{lem2} if we take $\mu = \infty$. Therefore, we omit the details.
	
	\textit{(2)} We are given that $v\in \langle\{s_i\,|\, l(s_i w)< l(w)\}\rangle$. Using \lemref{Ti invariant}, for $b\in B_w(\infty)$ we have 
	\begin{equation}\label{Tij in Bwinfty}
		T_{i_j}(b) \in B_w(\infty),\, \forall j \in \{1,2,\cdots,k\}.
	\end{equation} For all $b\in B_w(\infty)$ and non-negative integers $c_1,\cdots, c_k$, we need to prove that
	\begin{equation}\label{fi1..fik infty}
		f_{i_1}^{c_1}\cdots f_{i_k}^{c_k}(b_{\lambda}\otimes b)\in B_v(\lambda)\otimes B_w(\infty) \text{ if } f_{i_1}^{c_1}\cdots f_{i_k}^{c_k}(b_{\lambda}\otimes b)\neq 0.
	\end{equation} 	
 Using \eqref{Tij in Bwinfty}, the proof of \eqref{fi1..fik infty} is same as the proof of \eqref{fi1...fik} given in part \textit{(2)} of \lemref{lem2} if take $\mu = \infty$.	Therefore, we omit the details.
\end{proof}
From \thmref{identity result}, we have 
$$B_e(\lambda)\otimes B_w(\infty) = \coprod_{b\in B_w(\infty)^{\lambda}}T_{y_{w,b}^{\lambda}}(b_{\lambda}\otimes b)$$
For $v\in W$, let $s_{i_1}\cdots s_{i_k}$ be a reduced expression of $v_{\min}^{\lambda}$ in $W$. Define $u_j \in W$ inductively for $j=1,2,\cdots, k$ as defined in \subsecref{uj subsection}. For $b\in B_{w}(\infty)^{\lambda}$, define $u(b,v):= u_1$. Then we have the following theorem.
 \begin{theorem}\label{disjoint union infty}
 	Let $\lambda \in P^+,\,v,w\in W, $ and $s_{i_1}\cdots s_{i_k}$ be a reduced expression of $v_{\min}^\lambda$ in $W$. Assume that $v_{\min}^{\lambda}\in \langle\{s_i\,|\, l(s_i w)< l(w)\}\rangle$ then 
 	$$B_v(\lambda)\otimes B_w(\infty)\cong \coprod_{b\in B_w(\infty)^{\lambda}}B_{u(b,v)}(\infty;\lambda+wt(b)).$$ 
 \end{theorem}
 \begin{proof}
 	Using \remref{decom of lambda and infty}, we have 
 		\begin{equation}
 		B_v(\lambda)\otimes B_w(\infty) = \coprod_{b\in B_w(\infty)^\lambda} \mathcal{F}(b_{\lambda}\otimes b,v)
 	\end{equation} 
 	where $\mathcal{F}(b_{\lambda}\otimes b,v)$ denotes the connected component of $B_v(\lambda)\otimes B_w(\infty)$ containing $b_\lambda \otimes b$.	
 	Using \lemref{lem infty}, and $vW_{\lambda} = v_{\min}^{\lambda}W_{\lambda}$, we have
 	\begin{equation}
 		B_v(\lambda)\otimes B_w(\infty) = B_{v_{\min}^{\lambda}}(\lambda)\otimes B_{w}(\infty)
 		= T_{v_{\min}^{\lambda},\lambda,w} = T_{i_1}\cdots T_{i_k}(b_{\lambda} \otimes B_{w}(\infty))
 	\end{equation} 
 	The rest of the proof remains exactly the same as given in \thmref{disjoint union} by taking $\mu = \infty$ and replacing $B_{y_{w,b}^\lambda}(\lambda+wt(b))$ with $B_{y_{w,b}^\lambda}(\infty; \lambda+wt(b))$.
 \end{proof}
 \subsection{Proof of the \thmref{main result 2}} 
 
 $(1)\implies (3)$ Using \remref{extremal subset} and \remref{union are extremal} we have $B_v(\lambda)\otimes B_w(\infty)$ is an extremal subset of $B(\lambda)\otimes B(\infty)$.
 
 $(3)\implies (2)$ Using \lemref{A implies B}, we have $v_{\min}^{\lambda}\in \langle \{s_i\, |\, l(s_i w)< l(w)\}\rangle$.
 
 $(2)\implies(1)$ Let $ s_{i_1}\cdots s_{i_k}$ be a reduced expression of $v_{\min}^\lambda$ in $W$. Using \thmref{disjoint union infty}, we get that $B_v(\lambda)\otimes B_w(\infty)$ is isomorphic to a disjoint union of Demazure crystals.

 Let us explain our results with the help of the following example.
 \begin{example}
 	Let $\mathfrak{g} = \mathfrak{sl}_4(\mathbb{C}), \, I: = \{1,2,3\},$ and $\Delta = \{\alpha_1, \alpha_2, \alpha_3\}$, the set of simple roots of $\mathfrak{g}$. The Weyl group $W = \langle\{s_1,s_2,s_3\}\rangle$ and $\{\omega_1,\omega_2,\omega_3\}$ denotes the set of fundamental weights of $\mathfrak{g}$.
 	
 	 Let $\lambda = \omega_2\in P^+$, $v=s_2$, and $w= s_2s_1s_3s_2$ in $W$. Using \thmref{identity result}, $B_e(\omega_2)\otimes B_{s_2s_1s_3s_2}(\infty)$ is isomorphic to a disjoint union of Demazure crystals as given in \cite{Anthony}. 
 	$$\begin{array}{ll}
 		B_e(\omega_2)\otimes B_{s_2s_1s_3s_2}(\infty) &\cong B_{s_1 s_3}(\infty; \omega_2)\coprod B_{s_3 s_1 s_2}(\infty;\omega_2-\alpha_2)\coprod B_{s_2 s_1 s_2}(\infty;\omega_2-\alpha_1-\alpha_2)\\
 		& \coprod B_{s_2 s_3 s_2}(\infty;\omega_2-\alpha_2-\alpha_3)\coprod B_{s_2 s_1 s_3}(\infty;\omega_2-\alpha_1-\alpha_2-\alpha_3)\\
 		&\coprod B_{s_2 s_1 s_3 s_2}(\infty;\omega_2-\alpha_1-2\alpha_2-\alpha_3).
 	\end{array}$$
 Since $v = v_{\min}^\lambda = s_2$ and $l(s_2 w)<l(w)$. Using \thmref{main result 2}, $B_{s_2}(\omega_2)\otimes B_{s_2s_1s_3s_2}(\infty)$ is isomorphic to a disjoint union of Demazure crystals. Using \thmref{disjoint union infty}, we can explicitly describe the Demazure crystals that appear in the decomposition. 
 	 	$$\begin{array}{ll}
 	 	B_{s_2}(\omega_2)\otimes B_{s_2s_1s_3s_2}(\infty) &\cong B_{s_2s_1 s_3}(\infty; \omega_2)\coprod B_{s_2s_3 s_1 s_2}(\infty;\omega_2-\alpha_2)\coprod B_{s_2 s_1 s_2}(\infty;\omega_2-\alpha_1-\alpha_2)\\
 	 	& \coprod B_{s_2 s_3 s_2}(\infty;\omega_2-\alpha_2-\alpha_3)\coprod B_{s_2 s_1 s_3}(\infty;\omega_2-\alpha_1-\alpha_2-\alpha_3)\\
 	 	&\coprod B_{s_2 s_1 s_3 s_2}(\infty;\omega_2-\alpha_1-2\alpha_2-\alpha_3).
 	 \end{array}$$
 \end{example} 
 \section{Acknowledgements}
 The author would like to thank Sankaran Viswanath for many helpful discussions and useful comments. The author acknowledge the hospitality and excellent working conditions at the Institute of Mathematical Sciences, Chennai, India where a part of this work was done.  
	\bibliographystyle{plain}
	\bibliography{References}

@article {K1demazure,
	AUTHOR = {Kashiwara, Masaki},
	TITLE = {The crystal base and {L}ittelmann's refined {D}emazure
		character formula},
	JOURNAL = {Duke Math. J.},
	FJOURNAL = {Duke Mathematical Journal},
	VOLUME = {71},
	YEAR = {1993},
	NUMBER = {3},
	PAGES = {839--858},
	ISSN = {0012-7094,1547-7398},
	MRCLASS = {17B37 (17B10 17B67)},
	MRNUMBER = {1240605},
	MRREVIEWER = {Kailash\ C.\ Misra},
	DOI = {10.1215/S0012-7094-93-07131-1},
	URL = {https://doi.org/10.1215/S0012-7094-93-07131-1},
}

@article {Assaf,
	AUTHOR = {Assaf, Sami and Dranowski, Anne and Gonz\'alez, Nicolle},
	TITLE = {Extremal tensor products of {D}emazure crystals},
	JOURNAL = {Algebr. Represent. Theory},
	FJOURNAL = {Algebras and Representation Theory},
	VOLUME = {27},
	YEAR = {2024},
	NUMBER = {1},
	PAGES = {627--638},
	ISSN = {1386-923X,1572-9079},
	MRCLASS = {05E10 (17B10 17B37 20G42)},
	MRNUMBER = {4712432},
	DOI = {10.1007/s10468-023-10231-z},
	URL = {https://doi.org/10.1007/s10468-023-10231-z},
}

@article {Anthony,
	AUTHOR = {Joseph, Anthony},
	TITLE = {A decomposition theorem for {D}emazure crystals},
	JOURNAL = {J. Algebra},
	FJOURNAL = {Journal of Algebra},
	VOLUME = {265},
	YEAR = {2003},
	NUMBER = {2},
	PAGES = {562--578},
	ISSN = {0021-8693,1090-266X},
	MRCLASS = {17B67 (17B10)},
	MRNUMBER = {1987017},
	MRREVIEWER = {Nicol\'as\ Andruskiewitsch},
	DOI = {10.1016/S0021-8693(03)00028-0},
	URL = {https://doi.org/10.1016/S0021-8693(03)00028-0},
}

@article {L2,
	AUTHOR = {Littelmann, Peter},
	TITLE = {Paths and root operators in representation theory},
	JOURNAL = {Ann. of Math. (2)},
	FJOURNAL = {Annals of Mathematics. Second Series},
	VOLUME = {142},
	YEAR = {1995},
	NUMBER = {3},
	PAGES = {499--525},
	ISSN = {0003-486X,1939-8980},
	MRCLASS = {17B10 (17B67)},
	MRNUMBER = {1356780},
	MRREVIEWER = {Arun\ Ram},
	DOI = {10.2307/2118553},
	URL = {https://doi.org/10.2307/2118553},
}

@article {Knouo,
	AUTHOR = {Kouno, Takafumi},
	TITLE = {Decomposition of tensor products of {D}emazure crystals},
	JOURNAL = {J. Algebra},
	FJOURNAL = {Journal of Algebra},
	VOLUME = {546},
	YEAR = {2020},
	PAGES = {641--678},
	ISSN = {0021-8693,1090-266X},
	MRCLASS = {17B37 (05E05 05E14)},
	MRNUMBER = {4036671},
	MRREVIEWER = {Vladislav\ K.\ Kharchenko},
	DOI = {10.1016/j.jalgebra.2019.11.001},
	URL = {https://doi.org/10.1016/j.jalgebra.2019.11.001},
}

@book {Humphreys,
	AUTHOR = {Humphreys, James E.},
	TITLE = {Reflection groups and {C}oxeter groups},
	SERIES = {Cambridge Studies in Advanced Mathematics},
	VOLUME = {29},
	PUBLISHER = {Cambridge University Press, Cambridge},
	YEAR = {1990},
	PAGES = {xii+204},
	ISBN = {0-521-37510-X},
	MRCLASS = {20-02 (20F32 20F55 20G15 20H15)},
	MRNUMBER = {1066460},
	MRREVIEWER = {Louis\ Solomon},
	DOI = {10.1017/CBO9780511623646},
	URL = {https://doi.org/10.1017/CBO9780511623646},
}

@article {Kq-analogue,
	AUTHOR = {Kashiwara, M.},
	TITLE = {On crystal bases of the {$Q$}-analogue of universal enveloping
	algebras},
	JOURNAL = {Duke Math. J.},
	FJOURNAL = {Duke Mathematical Journal},
	VOLUME = {63},
	YEAR = {1991},
	NUMBER = {2},
	PAGES = {465--516},
	ISSN = {0012-7094,1547-7398},
	MRCLASS = {17B37 (17B10 17B67)},
	MRNUMBER = {1115118},
	MRREVIEWER = {Kailash\ C.\ Misra},
	DOI = {10.1215/S0012-7094-91-06321-0},
	URL = {https://doi.org/10.1215/S0012-7094-91-06321-0},
}

@article {L1,
	AUTHOR = {Littelmann, Peter},
	TITLE = {Paths and root operators in representation theory},
	JOURNAL = {Ann. of Math. (2)},
	FJOURNAL = {Annals of Mathematics. Second Series},
	VOLUME = {142},
	YEAR = {1995},
	NUMBER = {3},
	PAGES = {499--525},
	ISSN = {0003-486X,1939-8980},
	MRCLASS = {17B10 (17B67)},
	MRNUMBER = {1356780},
	MRREVIEWER = {Arun\ Ram},
	DOI = {10.2307/2118553},
	URL = {https://doi.org/10.2307/2118553},
}

@article {RS,
	AUTHOR = {Reiner, Victor and Shimozono, Mark},
	TITLE = {Key polynomials and a flagged {L}ittlewood-{R}ichardson rule},
	JOURNAL = {J. Combin. Theory Ser. A},
	FJOURNAL = {Journal of Combinatorial Theory. Series A},
	VOLUME = {70},
	YEAR = {1995},
	NUMBER = {1},
	PAGES = {107--143},
	ISSN = {0097-3165,1096-0899},
	MRCLASS = {05E05 (05E10 14M15 20C30)},
	MRNUMBER = {1324004},
	MRREVIEWER = {Witold\ Kra\'skiewicz},
	DOI = {10.1016/0097-3165(95)90083-7},
	URL = {https://doi.org/10.1016/0097-3165(95)90083-7},
}

@article {LLM,
	AUTHOR = {Lakshmibai, Venkatramani and Littelmann, Peter and Magyar,
	Peter},
	TITLE = {Standard monomial theory for {B}ott-{S}amelson varieties},
	JOURNAL = {Compositio Math.},
	FJOURNAL = {Compositio Mathematica},
	VOLUME = {130},
	YEAR = {2002},
	NUMBER = {3},
	PAGES = {293--318},
	ISSN = {0010-437X,1570-5846},
	MRCLASS = {14M15 (14L30)},
	MRNUMBER = {1887117},
	MRREVIEWER = {E.\ A.\ Tevel\"ev},
	DOI = {10.1023/A:1014396129323},
	URL = {https://doi.org/10.1023/A:1014396129323},
}

@article {Anthonydemazureformula,
	AUTHOR = {Joseph, A.},
	TITLE = {On the {D}emazure character formula},
	JOURNAL = {Ann. Sci. \'Ecole Norm. Sup. (4)},
	FJOURNAL = {Annales Scientifiques de l'\'Ecole Normale Sup\'erieure.
	Quatri\`eme S\'erie},
	VOLUME = {18},
	YEAR = {1985},
	NUMBER = {3},
	PAGES = {389--419},
	ISSN = {0012-9593},
	MRCLASS = {17B10 (14M15 20G05)},
	MRNUMBER = {826100},
	MRREVIEWER = {James\ E.\ Humphreys},
	URL = {http://www.numdam.org/item?id=ASENS_1985_4_18_3_389_0},
}
	\end{document}